\documentclass[12pt]{amsart}

\usepackage{amsfonts,amssymb,stmaryrd,amscd,amsmath,latexsym,amsbsy}

\newtheorem{theorem}{Theorem}[section]
\newtheorem{lemma}[theorem]{Lemma}
\newtheorem{proposition}[theorem]{Proposition}
\newtheorem{corollary}[theorem]{Corollary}
\theoremstyle{definition}
\newtheorem{definition}[theorem]{Definition}

\newtheorem{example}[theorem]{Example}
\newtheorem{question}[theorem]{Question}

\newcommand{\Tr}{\text{Tr}}
\newcommand{\id}{\text{id}}

\newcommand{\End}{\text{End}}
\newcommand{\Hom}{\text{Hom}}

\newcommand{\Rep}{\text{Rep}}

\newcommand{\op}{{\text{op}}}

\newcommand{\g}{\mathfrak{g}}

\newcommand{\C}{\mathcal{C}}

\newcommand{\ben}{\begin{enumerate}}
\newcommand{\een}{\end{enumerate}}

\newcommand{\Vect}{{\text{Vect}}}

\newcommand{\be}{{\bf 1}}

\theoremstyle{plain}

\newtheorem*{sol}{Solution}

\theoremstyle{definition}

\theoremstyle{remark}

\newcommand{\solu}[1]{\begin{sol}{\bf (\ref{#1})}}

\def\g{\mathfrak{g}}

\def\C{\mathcal{C}}
\def\D{\mathcal{D}}

\def\Vect{\mathrm{Vect}}

\def\N{\mathcal{N}}
\def\M{\mathcal{M}}

\def\End{\mathrm{End}}
\def\Hom{\mathrm{Hom}}

\def\Gr{\mathop{Gr}}

\def\Vec{\mathrm{Vec}}

\def\k{\mathbf{k}}

\def\Rep{\mathop{\mathrm{Rep}}\nolimits}

\newcommand{\CM}  {{}_{\C}\M}
\newcommand{\CN}  {{}_{\C}\N}
\newcommand{\CCm}  {{}_{\C}\C}

\newcommand{\IHom}        {\operatorname{\underline{\mathsf{Hom}}}}

\def\calm             {{\mathcal M}}
\def\calmopp       {{\mathcal M^{\mathrm{op}}_{\phantom|}}}

\def\Sl               {{\rm S}^{\rm l}}
\def\Sr               {{\rm S}^{\rm r}}
\def\Sk               {{\rm S}_{\k}}
\newcommand{\Fun}         {\operatorname{\mathrm{Fun}}}

\begin{document}

\title{Eigenvalues of the squared antipode in finite dimensional weak Hopf algebras}

\author{Pavel Etingof}
\address{Department of Mathematics, Massachusetts Institute of Technology,
Cambridge, MA 02139, USA}
\email{etingof@math.mit.edu}

\address{Gregor Schaumann, 
Faculty of Mathematics, University of Vienna, Austria}
\email{gregor.schaumann@univie.ac.at}

\maketitle

\vskip .05in 
\centerline{\bf With an appendix by Gregor Schaumann}
\vskip .05in

\begin{abstract} We extend Schaumann's theory of pivotal structures on fusion categories matched to a module category and of module traces to the case of non-semisimple tensor categories, and use it to study eigenvalues of the squared antipode $S^2$ in weak Hopf algebras. In particular, we diagonalize $S^2$ for semisimple weak Hopf algebras in characteristic zero, generalizing the result of Nikshych in the pseudounitary case. We show that the answer depends only on the Grothendieck group data of the pivotalizations of the categories involved and the global dimension of the fusion category (thus, all eigenvalues belong to the corresponding number field). On the other hand, we study the eigenvalues of $S^2$ on the non-semisimple weak Hopf algebras attached to dynamical quantum groups at roots of $1$ defined by D. Nikshych and the author in 2000, and show that they depend nontrivially on the continuous parameters of the corresponding module category. We then compute these eigenvalues as rational functions of these parameters. The paper also contains an appendix by G. Schaumann discussing the connection between our generalization of module traces and the notion of an inner product module category introduced in \cite{Sch2}.  
\end{abstract} 

\section{Introduction} 

Let $\k$ be an algebraically closed field. Let $\C$ be a finite tensor category over $\k$, and 
$\M$ a semisimple indecomposable (left) module category over $\C$. Let $A$ be a commutative semisimple algebra 
such that $\M=A-{\rm mod}$, i.e., $A=\oplus_{i\in I}\k$, where $I$ is the set labeling the simple objects $M_i$ of $\M$. 
Then the action of $\C$ on $\M$ gives rise to a tensor functor $F: \M\to A-{\rm bimod}$, 
and the algebra $H:={\rm End}_\k F$ is a regular biconnected weak Hopf algebra with base $A$ (where ${\rm End}_\k F={\rm End}\overline{F}$ 
is the endomorphism algebra of the composition $\overline{F}$ of $F$ with the forgetful functor from $A$-bimodules to vector spaces).
Conversely, if $H$ is a regular biconnected weak Hopf algebra with commutative base $A$ then $\C=\Rep H$ is a finite tensor category, and 
$\M=A-{\rm mod}$ is an indecomposable semisimple $\C$-module category, and these assignments are inverse to each other
(see \cite{N1} and \cite{EGNO}, Subsection 7.23). 

Recall that $H$ has a natural automorphism -- the squared antipode $S^2$, which is semisimple in characteristic zero (\cite{N1}, Corollary 5.15).
Therefore, the characteristic polynomial of $S^2$ is an invariant 
of the pair $(\C,\M)$. The goal of this paper is to study this invariant. 

If $\M={\rm Vec}$ is the category of vector spaces, then $H$ is a finite dimensional Hopf algebra, and by Radford's theorem 
$S^2$ has finite order dividing $2\dim H$. In particular, the eigenvalues of $S^2$ are roots of unity of order dividing $2\dim H$.
A natural question is whether this statement extends to the case of weak Hopf algebras. 

Consider first the case when $H$ is semisimple (i.e., $\C$ a fusion category) and ${\rm char}(\k)=0$. In this case, it follows 
from Ocneanu rigidity (\cite{EGNO}, Subsection 9.1) that the eigenvalues of $S^2$ are algebraic numbers, but getting 
more detailed information about them requires further study. When $\C$ is pseudounitary, 
this problem was solved by D. Nikshych in \cite{N2}. Namely, it is shown in \cite{N2} that in this case the eigenvalues of $S^2$,  
while in general not roots of unity, can be computed explicitly, and in fact depend only on the combinatorial information -- 
the $\Bbb Z_+$-module ${\rm Gr}(\M)$ over the Grothendieck ring ${\rm Gr}(\C)$ of $\C$. More precisely, the
characteristic polynomial of $S^2$ has the form 
\begin{equation}\label{eigfor}
\chi(z)=\prod_{i,j,k,l\in I}(z-\lambda_{ijkl})^{n_{ijkl}}, 
\end{equation} 
with 
$$ 
\lambda_{ijkl}:=\frac{m_jm_l}{m_im_k},\quad
n_{ijkl}:=\sum_{r\in J}N_{ri}^jN_{rl}^k,
$$
where $m_i$ are the Frobenius-Perron dimensions of $M_i$, $J$ labels the simple objects $X_r$ of $\C$, and 
$$
N_{ri}^j:=\dim {\rm Hom}(X_r\otimes M_i,M_j). 
$$

If a fusion category $\C$ is not pseudounitary, the situation is more complicated. So let us first assume that
the squared antipode is the conjugation by a trivial group-like element, i.e., 
$S^2(x)=gxg^{-1}$, where $g=yS(y)^{-1}$ and $y\in A_s\cong A$ is an invertible element of the source 
base of $H$ (this holds in all the examples we know). In this case, following Schaumann (\cite{Sch1}), we will say that the pivotal structure on $\C$ defined by $g$ is {\it matched} to the module category 
$\M$; we show that our definition is equivalent to that of \cite{Sch1}. In the pseudounitary case it is shown in \cite{N2} that this is automatic. 
In general, in the matched case it follows from \cite{Sch1} that the characteristic polynomial of $S^2$ is still given by formula \eqref{eigfor}, 
but with $m_i$ now being nonzero numbers such that for all $r,i$ we have 
$$
\sum_{j\in I} N_{ri}^j m_j=\dim(X_r)m_i,
$$
where $\dim(X_i)$ is the dimension of $X_i$ in the pivotal structure defined by $g$
(i.e., $N_r\bold m=\dim(X_r)\bold m$, where $N_r:=(N_{ri}^j)$ and $\bold m=(m_i)$). 
More precisely, the numbers $m_i$ are uniquely determined up to scaling (as the dimension character
$X\mapsto \dim(X)$ occurs with multiplicity one in the module ${\rm Gr}(\M)\otimes \k$ over ${\rm Gr}(\C)\otimes \k$), 
and are all nonzero. We show that this answer depends only on the global dimension $\dim(\C)$ of $\C$ and ${\rm Gr}(\M)$ as a $\Bbb Z_+$-module over ${\rm Gr}(\C)$. 

For a general fusion category $\C$ (i.e., without assuming the existence of a matched pivotal structure) 
we obtain a slightly more complicated but similar answer (see Theorem \ref{main}). This answer depends only on $\dim(\C)$ and ${\rm Gr}(\widetilde{\M})$ as a $\Bbb Z_+$-module over ${\rm Gr}(\widetilde{\C})$, where $\widetilde{\C}$ and $\widetilde{\M}$ are the pivotalizations of $\C,\M$. 

We also extend these results to the non-semisimple case (for ${\rm char}(\k)=0$). In this case, 
Ocneanu rigidity is no longer valid, and module categories over $\C$ (as well as $\C$ itself) 
may depend on continuous parameters. So one may wonder if the eigenvalues of $S^2$ 
may also depend on these parameters (and hence fail to be algebraic numbers, in general). 
We again consider the setting when $\M$ is matched to the pivotal structure on $\C$, 
and show that if the dimension character occurs with multiplicity (at most) one in the  
module ${\rm Gr}(\M)\otimes \k$ over ${\rm Gr}(\C)\otimes \k$ 
(e.g., if ${\rm Gr}(\M)\otimes \k$ is a cyclic module over ${\rm Gr}(\C)\otimes \k$)
then this does not happen and the eigenvalues of $S^2$ are algebraic numbers
which we compute explicitly. Namely, we show that the characteristic polynomial of $S^2$ 
is still given by \eqref{eigfor}, but with 
$$ 
n_{ijkl}:=\sum_{q,r\in J}N_{qi}^jC_{qr}N_{rl}^k,
$$
where $C=(C_{qr})$ is the Cartan matrix of $\C$. A similar result holds in positive characteristic. 

On the other hand, we consider the example of dynamical quantum groups at roots of $1$ 
introduced in \cite{EN}. In this case, the pivotal structure is matched to 
the corresponding module category, but the dimension character turns out to occur more than once
in the module ${\rm Gr}(\M)\otimes \k$. As a result, the particular choice 
of the eigenvector $\bold m=(m_i)$ of $X_r$ with eigenvalues given by this character actually depends 
on the continuous parameter of the module category (an element $\Lambda$ of the maximal torus of the corresponding simple algebraic group), and 
we obtain a formula for the eigenvalues of $S^2$ as rational functions of $\Lambda$.  

The organization of the paper is as follows. In Section 2 we recall the results of Schaumann on pivotal structures matched to 
module categories and extend them to the non-semisimple case. We also give two other equivalent definitions of this notion. In Section 3 we prove the formula for the characteristic polynomial of $S^2$ in the case when the dimension character occurs at most once in 
${\rm Gr}(\M)\otimes \k$, and use it to obtain a formula for the characteristic polynomial of $S^2$ for arbitrary semisimple $\C,\M$ when 
${\rm char}(\k)=0$. In Section 4 we consider the example of dynamical quantum groups and compute the eigenvalues of $S^2$ in this example as rational functions of the parameters. 

We note that the notion of a module trace has been generalized by G. Schaumann to the case when $\C$ is a finite tensor category 
and $\M$ an exact (not necessarily semisimple) $\C$-module category, \cite{Sch2}; namely, he introduced the notion of an inner product module category 
(\cite{Sch2}, Definition 5.2). Specifically, $\M$ is an inner product module category if and only if the pivotal structure on $\C$ extends to a 
pivotal structure on the category $\C_{\C\oplus \M}^{*\rm op}$. G. Schaumann showed that this generalization coincides with ours when $\C$ is unimodular (i.e., the distinguished invertible object of $\C$ is $\bold 1$), up to switching the pivotal structure to the conjugate one, and also when $\C_\M^*$ is unimodular, 
but they differ in general. The paper contains an appendix by G. Schaumann explaining the exact 
relationship between these two generalizations. 

{\bf Acknowledgements.} P. Etingof thanks D. Nikshych and V. Ostrik for useful discussions and comments on the draft of this paper,
and in particular to V. Ostrik for providing the reference \cite{Sch1}. He is also very grateful to G. Schaumann for comments on the draft of the paper and for contributing an appendix. The work of P. Etingof was partially supported by the NSF grant DMS-1502244.  The work of G. Schaumann was partially supported by the stand-alone project P 27513-N27 of the Austrian Science Fund.

\section{Pivotal structures matched to a module category} 

We start with giving a straightforward generalization of the results of \cite{Sch1} to the setting of non-semisimple 
tensor categories in any characteristic. So this subsection is essentially an exposition of the results of \cite{Sch1} 
in a slightly more general setting using a somewhat different approach. 

\subsection{Definitions and examples of matched pivotal structures} 

Let $\C$ be a finite tensor category over an algebraically closed field $\k$. Let $\M$ be a finite semisimple indecomposable module category over $\C$.
In other words, $\M$ is a finite semisimple category and we are given a tensor functor $F: \C\to \End \M$. Let $\C_\M^*={\rm Fun}_\C(\M,\M)$ be the dual category. 
(For basics on tensor categories, module categories and dual categories, see \cite{EGNO}). 

Let $a:{\rm Id}\to **$ be a tensor isomorphism, i.e., a pivotal structure on $\C$. Then we can canonically define an invertible object 
$\chi(a,\M)\in \C_\M^*$ as follows. For $V\in \End \M$, we may naturally identify $V$ with $V^{**}$, so for $X\in \C$ we have a natural isomorphism $F(X^{**})\cong F(X)^{**}\cong F(X)$. Hence we have a canonical identification $X^{**}\otimes M\cong X\otimes M$, $M\in \M$. Thus the pivotal structure $a$ 
defines a functorial automorphism 
$$
a\otimes 1: X\otimes M\to X^{**}\otimes M=X\otimes M,
$$ 
i.e., an invertible module functor $\M\to \M$ (with the underlying additive functor being the identity). But such a functor is nothing but an invertible object of the dual category $\C_\M^*$, which we denote by $\chi(a,\M)$. 

\begin{definition}\label{matc} We say that $a$ is matched to $\M$ if $\chi(a,\M)=\bold 1$. 
\end{definition}  

This definition can be extended to not necessarily indecomposable categories in an obvious way. 
Namely, if $\M=\oplus_i \M_i$ and $\M_i$ are indecomposable then $a$ is matched to $\M$ if and only 
if it is matched to each $\M_i$.  

One can reformulate this definition in terms of weak Hopf algebras. (For basics on weak Hopf algebras, see \cite{N1} and references therein). 
Namely, let $A$ be a commutative semisimple algebra such that $\M=A-{\rm mod}$, so that $\End \M=A-{\rm bimod}$. 
Let $F: \C\to A-{\rm bimod}$ be the corresponding tensor functor.
Recall that to $\C,\M$ we can associate the biconnected regular Frobenius weak Hopf algebra 
$H:=\End_\k F={\rm End}\overline{F}$ with bases $A_s,A_t\cong A$, and $\C=\Rep H$ (here $\overline{F}$ 
is the composition of $F$ with the forgetful functor $A-{\rm bimod}\to \Vec$). 
A pivotal structure $a$ on $\C$ then gives rise to a grouplike element 
$g_a\in H$ such that $g_a x g_a^{-1}=S^2(x)$, $x\in H$, 
called the {\it pivotal} element. Such an element defines 
an invertible $H^*$-module (\cite{N1}, Section 4), 
which is exactly the object $\chi(a,\M)\in \Rep H^*_{\rm op}=\C_\M^*$ (\cite{ENO}, p.590; \cite{O}, Theorem 4.2; \cite{EGNO}, Example 7.12.26).
\footnote{Note that in Example 7.12.26 of \cite{EGNO}, there is a misprint: 
$H^*$ should be $H^*_{\rm op}$ or $H^{*{\rm cop}}$.} 
Thus, we have the following proposition, giving an equivalent definition 
of a pivotal structure matched to $\M$: 

\begin{proposition}\label{weakhopf} The pivotal structure $a$ is matched to $\M$ if and only if the corresponding pivotal element $g_a$ is a trivial grouplike element of $H$, i.e., 
$g_a=yS(y)^{-1}$ where $y\in A_s$ is invertible. 
\end{proposition}  

Another definition of a matched pivotal structure (in fact, several equivalent definitions) is given by G. Schaumann in \cite{Sch1}.
Namely, fix a collection of numbers $\bold m=(m_i,i\in I)$ in $\k$, and for any $M\in \M$ and endomorphism $f: M\to M$, 
define its trace
$$
{\rm Tr}^{\bold m}_M(f):=\sum_{i\in I}m_i {\rm Tr}(f|_{\Hom(M_i,M)}).
$$

\begin{definition} (\cite{Sch1}) 
The trace ${\rm Tr}^{\bold m}$ is said to be a 
{\it module trace} on $\M$ with respect to a pivotal structure $a$ on $\C$
if for any $X\in \C$, $M\in M$ and morphism $f: X\otimes M\to X\otimes M$ 
one has 
$$
\Tr^{\bold m}_{X\otimes M}(f)=\Tr^{\bold m}_M(\Tr_X(f)),
$$
where ${\rm Tr}_X(f): M\to M$ is the composition 
$$
M\xrightarrow{{\rm coev}_X\otimes 1}X\otimes X^*\otimes M\xrightarrow{a_X\otimes f} X^{**}\otimes X^*\otimes M\xrightarrow{{\rm ev}_{X^*}\otimes 1} M.
$$
\end{definition} 

\begin{proposition}\label{scha1} The pivotal structure $a$ is matched to $\M$ if and only if $\M$ admits a nonzero module trace with respect to $a$. In this case, the module trace is unique up to scaling. 
\end{proposition} 

\begin{proof} Suppose that $a$ is matched to $M$, and let $y\in A=A_s$ be such that $g_a=yS(y)^{-1}$ in the corresponding weak Hopf algebra $H$ (it exists by Proposition \ref{weakhopf}). 
Recall that $A=\oplus_{i\in I}\k$, and let $m_i\in \k$ be the projection of $y$ onto the $i$-th summand. We claim that $\Tr^{\bold m}$ 
is a module trace. Indeed, let $X\in \C$ and $f: X\otimes M_i\to X\otimes M_i$ be a morphism. Then we have 
$$
\Tr^{\bold m}(f)=\sum_{j\in I}\beta_{ij}m_j,
$$ 
where 
$$
\beta_{ij}:={\rm Tr}(f|_{\Hom(M_j,X\otimes M_i)}).
$$
It is clear that it suffices to check the module trace property for simple $M$, 
so it is enough to check that 
$$
\Tr_X(f)=m_i^{-1}\sum_{j\in I}\beta_{ij}m_j.
$$ 
Let $F(X)$ denote the $A$-bimodule corresponding to $X$. Then 
$$
F(X)=\oplus_{i,j\in I}\Hom(M_j,X\otimes M_i).
$$
Thus 
$$
X\otimes M_i=\oplus_{j\in I}\Hom(M_j,X\otimes M_i)\otimes M_j
$$
and $g_a=yS(y)^{-1}$ acts on the summand $\Hom(M_j,X\otimes M_i)\otimes M_j$ by multiplication by $m_jm_i^{-1}$. 
Thus, 
$$
\Tr_X(f)=\Tr(f\circ (g\otimes 1))=m_i^{-1}\sum_{j\in I}\beta_{ij}m_j,
$$
as desired. 

Conversely, suppose that ${\rm Tr}^{\bold m}$ is a nonzero module trace on $\M$. 
We claim that $m_i\ne 0$ for all $i$. Indeed, if $m_i=0$ for some $i$ then for any $X\in \C$, the trace 
of any endomorphism of $X\otimes M_i$ is zero. But any object of $\M$ is a direct summand of $X\otimes M_i$ for some $X$, which implies that $m_j=0$ for all $j$. 

Thus, we have an invertible element $y\in A=A_s$ whose projection to the $i$-th summand of $A$ is $m_i$. Let $g=yS(y)^{-1}\in H$. 
It is clear that $g_a$ and $g$ act naturally on $\Hom(M_j,X\otimes M_i)$ for each $i,j$, and
the module trace property implies that 
$$
\Tr((g_a-g)\circ f|_{\Hom(M_j,X\otimes M_i)})=0
$$
for any $X\in \C$ and $f: X\otimes M_i\to X\otimes M_i$. This implies that $g_a-g$ acts by zero on $\Hom(M_j,X\otimes M_i)$ for all $X,i,j$, i.e., $g_a=g=yS(y)^{-1}$, as desired.   

The uniqueness of the module trace up to scaling (when it exists) follows from the fact proved above that for a module trace, $m_i\ne 0$ for all $i$
(as in \cite{Sch1}). 
\end{proof} 

Now observe that any collection of numbers $m_i\ne 0$ gives rise to a functorial isomorphism $\phi^{\bold m}_{MN}:\Hom(M,N)\cong \Hom(N,M)^*$ for $M,N\in \M$ 
defined by the pairing $\Hom(M,N)\otimes \Hom(N,M)\to \k$ given by 
$$
\langle f, h\rangle:=\Tr^{\bold m}_N(f\circ h) 
$$
(note that $\Tr^{\bold m}_N(f\circ h)=\Tr^{\bold m}_M(h\circ f)$). Moreover, any functorial isomorphism $\Hom(M,N)\cong \Hom(N,M)^*$ is uniquely obtained in this way. 

\begin{definition} A functorial isomorphism 
$$
\phi_{MN}: \Hom(M,N)\to \Hom(N,M)^*
$$ 
is said to be $\C$-balanced (for a given pivotal structure $a: X\to X^{**}$ on $\C$) if for each $X\in \C$ the composite isomorphism
$$
\Hom(X\otimes M,N)\xrightarrow{\phi_{X\otimes M,N}}\Hom(N,X\otimes M)^*\cong \Hom(X^*\otimes N,M)^*
$$
$$
\xrightarrow{\phi_{M,X^*\otimes N}^{-1}} \Hom(M,X^*\otimes N)\cong \Hom(X^{**}\otimes M,N)\xrightarrow{\circ(a_X\otimes 1)} \Hom(X\otimes M,N)
$$
is the identity. 
\end{definition} 

\begin{proposition}\label{scha2} (\cite{Sch1}, Theorem 4.2) The isomorphism $\phi^{\bold m}$ is $\C$-balanced if and only if $\Tr^{\bold m}$ is a module trace. 
\end{proposition} 

\begin{proof} The proof is the same as the proof of Theorem 4.2 in \cite{Sch1}.  
\end{proof} 

Thus, thanks to Propositions \ref{weakhopf}, \ref{scha1}, \ref{scha2} we have four equivalent definitions of a pivotal structure matched to a module category. 

\begin{example} 1. Let $\C=\Rep H$, where $H$ is a finite dimensional Hopf algebra over $\k$, 
and $\M=\Vec$, so that $F$ is the corresponding fiber functor. Then
a pivotal structure $a$ is defined by a grouplike element $g_a\in H$ such that $g_a xg_a^{-1}=S^2(x)$, and 
$\C_\M^*=\Rep H^*_{\rm op}$, so $\chi(a,\M)\in H$ is a grouplike element. 
Recall that an $H^*_{\rm op}$-module $V$ gives rise to the $\C$-linear functor 
$\Vec\to \Vec$ given by $F_V(\k)=V$ with  
the map $X\otimes V\to X\otimes V$ for $X\in \C$ 
defined by $x\otimes v\mapsto \rho(v)(x\otimes 1)$, where 
$\rho:V\to H\otimes V$ is the left $H$-comodule structure 
attached to the right $H^*$-module structure on $V$. 
Therefore, we have $\chi(a,\M)=g_a$. Thus, $a$ is matched to 
$\M$ if and only if $g_a=1$. In this case $S^2={\rm Id}$, so $a$ is the canonical 
spherical structure, and $\C$ is semisimple if ${\rm char}(\k)=0$ by the Larson-Radford theorem. 
In positive characteristic $\C$ need not be semisimple (e.g., $H$ can be a group algebra).   

2. If $\M=\C$ for $\C$ semisimple then it is easy to see that $\chi(a,\M)=\bold 1$ for any pivotal structure $a$, so all pivotal structures are automatically matched to $\M$ (and the corresponding module trace is just the left trace in $\C$). 

3. If $\C$ is semisimple and pseudounitary over $\k=\Bbb C$ and $a$ is the canonical spherical structure on $\C$, then $a$ 
is matched to $\M$, since $S^2$ is given by conjugation by a trivial grouplike element, as shown in \cite{N2}. 

4. (\cite{Sch1}, Example 3.13). Let $\Vec_G^\omega$ be the category of $G$-graded finite dimensional vector spaces over $\k$ with 
associativity defined by a 3-cocycle $\omega$ on $G$. Then pivotal structures on $\C$ correspond to characters $\kappa: G\to \k^\times$. 
Also, indecomposable module categories over $\C$ are $\M(H,\psi)$, where $H\subset G$ is a subgroup and $\psi$ a 2-cochain on $H$ such that $d\psi=\omega|_H$
(\cite{EGNO}, Example 9.7.2). It is not hard to show that $a$ is matched to $\M(H,\psi)$ if and only if $\kappa|_H=1$ (as shown in \cite{Sch1}). 

5. (\cite{Sch1}, Example 3.14). Let $\C$ be a fusion category, and $\widetilde{\C}$ be the pivotalization of $\C$,  \cite{EGNO}, Definition 7.21.9. 
Thus we have a forgetful tensor functor $\widetilde{\C}\to \C$, hence $\C$ becomes a $\widetilde{\C}$-module category, and the dual category is 
the semidirect product $\Vec_{\Bbb Z/2}\ltimes \C$. Let $a$ be the canonical pivotal structure on $\widetilde{\C}$. Then $\chi(a,\C)$ 
is the invertible object of $\Vec_{\Bbb Z/2}$ corresponding to the generator $1\in \Bbb Z/2$. Thus, $a$ is not matched to $\C$. 

6. If $F:\C\to \D$ is a pivotal tensor functor between pivotal finite tensor categories and the pivotal structure of $\D$ is matched to $\M$ then 
the pivotal structure of $\C$ is also matched to $\M$. 
\end{example} 

\subsection{Properties of $m_i$}

Let $X_r$ be the simple objects of $\C$ and $N_r$ be the matrix of multiplication by $X_r$ in the basis $(M_i)$. 
Assume that $\C$ is equipped with a pivotal structure $a$ matched to $\M$, and let $\bold m=(m_i)$ be the corresponding vector. 

\begin{proposition}\label{eigenvec} (\cite{Sch1}) One has
$$
N_r\bold m=\dim(X_r)\bold m.
$$
\end{proposition}  

\begin{proof} 
Consider the module trace condition for the identity morphism:  
$$
\Tr^{\bold m}_{X_r\otimes M_i}({\rm Id})=\Tr^{\bold m}_{M_i}({\rm Tr}_{X_r}({\rm Id}))
$$
for $X\in \C, M\in \M$. This yields
$$
\sum_{j\in I}N_{ri}^jm_j=\dim(X_r)m_i,
$$
or in the matrix form $N_r\bold m=\dim(X_r)\bold m$, as desired. 
\end{proof} 

Consider now the Grothendieck group ${\rm Gr}(\M)\otimes \k$ as a module over ${\rm Gr}(\C)\otimes \k$.
Note that in characteristic zero this is a semisimple module (even for non-semisimple $\C$), since for $\k=\Bbb C$ it has a positive definite invariant Hermitian form in which $M_i$ are orthonormal. Using Proposition \ref{eigenvec}, we obtain the following proposition. 

\begin{proposition}\label{mult1} 
If the pivotal structure $a$ is matched to $\M$ then the dimension character $X\mapsto \dim(X)$ occurs in the decomposition of 
the ${\rm Gr}(\C)\otimes \k$-module ${\rm Gr}(\M)\otimes \k$, and $\bold m$ is an eigenvector corresponding to this character. 
Moreover, if this character occurs with multiplicity $1$, then this determines $\bold m$ up to a scalar.
\end{proposition} 

\begin{example} If $\C$ is pseudounitary then the dimension character occurs in ${\rm Gr}(\M)\otimes \k$ 
with multiplicity $1$ by the Frobenius-Perron theorem, so $m_i={\rm FPdim}(M_i)$, as shown in \cite{N2}. 
\end{example} 

\subsection{Tensor isomorphisms matched to a module category} 
Let us generalize Definition \ref{matc} to more general tensor isomorphisms. 
Let $r\in \Bbb Z$ and $b_X: X\to X^{(2r)}$ be a functorial tensor isomorphism (where $X^{(2r)}$ denotes the $2r$-th dual of $X$). Then we can define an invertible object $\chi(b,\M)\in \C_\M^*$ similarly to the above.

\begin{definition} We say that $b$ is matched to $\M$ if $\chi(b,\M)=\bold 1$. 
\end{definition} 

In the language of weak Hopf algebras, this means that $S^{2r}(x)=h_bxh_b^{-1}$, where $h_b\in H$ is a grouplike element, and the matched condition means that $h_b=zS(z)^{-1}$ is a trivial grouplike element (for some invertible $z\in A_s$).  
 
It is clear that 
$$
\chi(b_1\circ b_2,\M)=\chi(b_1,\M)\otimes \chi(b_2,\M)=\chi(b_2,\M)\otimes \chi(b_1,\M)
$$ 
(as $b_1\circ b_2=b_2\circ b_1$) so if $b_1,b_2$ are matched to $\M$ then so is $b_1\circ b_2$. 

\begin{example}\label{vost} 1. (V. Ostrik) Let $b_X: X\to X$ be a tensor automorphism of the identity 
functor of $\C$ (i.e., $r=0$). 
Then $b$ canonically defines an invertible object $Z_b$ in the Drinfeld center ${\mathcal Z}(\C)$ 
of $\C$ which is $\bold 1$ as an object of $\C$ (so $b_X$ is the corresponding isomorphism between $Z_b\otimes X$ and $X\otimes Z_b$). It follows from the definition that $\chi(b,\M)$ 
is the image of $Z_b$ in $\C_\M^*$. 

2. It is easy to show that $b$ is matched to $\M$ if and only if there is an automorphism $b_\M$ of the identity functor of $\M$
compatible with $b$ (which is necessarily unique and also has order $d$; in particular, $d\ne 0$ in $\k$). In other words, $b$ defines a faithful 
$\Bbb Z/d$-grading on $\C$, and $b_\M$ defines a compatible faithful $\Bbb Z/d$-grading on $\M$. 
\end{example} 

Now let $a,a'$ be two pivotal structures on $\C$ matched to $\M$. Then $a'=a\circ b$, where $b$ is a tensor automorphism 
of the identity fiunctor of $\C$ matched to $\M$. Thus there is a compatible automorphism $b_\M$ 
of the identity functor of $\M$. Let $b_i$ be the scalar by which $b_\M$ acts on $M_i$ (a root of unity of order dividing $d$). Then 
the numbers $m_i'$ corresponding to the pivotal structure $a'$ are related to the numbers $m_i$ corresponding to $a$
by the formula $m_i'=m_ib_i$.  

\begin{example} Let $\C$ be a finite tensor category and $\M$ be a semisimple indecomposable $\C$-module category. 
Let $D$ be the dual distinguished invertible object of $\C$ and $\iota: X\to D^{-1}\otimes X^{****}\otimes D$ be the categorical Radford isomorphism 
(see \cite{EGNO}, Theorem 7.19.1). Let $d$ be the least common multiple of the order order of $D$ and the order of the (dual) distinguished invertible object 
of $\C_\M^*$. Then the tensor isomorphism $b:=\iota^d: X\to X^{(4d)}$ is matched to $\M$. Indeed, by the Radford $S^4$ formula for weak Hopf algebras (\cite{N1}), the operator $S^{4d}$   
is given by conjugation by a trivial grouplike element. 

In particular, if $\C$ is semisimple then $d=1$ (\cite{N1}), so $\iota$ is matched to $\M$. 
In other words, in a semisimple weak Hopf algebra $S^4(x)=hxh^{-1}$ where $h$ is a trivial grouplike element, 
as shown in \cite{N2}. Thus, if $a$ is a spherical structure then $\chi(a,M)$ has order $\le 2$
(as $a^2=\iota$).  
\end{example} 

\subsection{The dual pivotal structure on the dual category} 
\label{dps} It is shown in \cite{Sch1} that if $\C$ is a finite tensor
category with pivotal structure $a$ and $\M$ a semisimple
indecomposable $\C$-module category such that $a$ is matched to $\M$
then the dual category $\C_\M^*$ acquires a canonical "dual" pivotal
structure $\overline{a}$, also matched to $\M$, with the same module
trace. This gives rise to a pivotal structure on $\C_\M^{*{\rm op}}$
which we will also denote by $\overline{a}$. An easy way to see this
is to note that since in the corresponding weak Hopf algebra $H$, one
has $S^2(x)=g_axg_a^{-1}$ where $g_a=yS(y)^{-1}$, a similar formula
holds in $H^*$ (\cite{N1}, Proposition 4.8). 

For example, if $\M=\C$ for $\C$ semisimple, we get a
dual pivotal structure $\overline{a}$ on $\C$ such that the left
dimensions of objects with respect to $a$ are equal to their right
dimensions with respect to $\overline{a}$, and vice versa. Namely, one
has $\overline{a}=\iota\circ a^{-1}$.  Thus, if $a$ is matched to a
module category $\N$ then so is $\overline{a}$, and $a$ is a spherical
structure if and only if $a=\overline{a}$. In particular, for
$\k=\Bbb C$, since the left and right dimensions are complex
conjugate, we get that the left dimensions with respect to $a$ are
complex conjugate to the left dimensions with respect to
$\overline{a}$ (as shown in \cite{Sch1}).
 
\begin{example} Let $\C={\mathcal E}\boxtimes {\mathcal E}^{\rm op}$ and $\M={\mathcal E}$ for a fusion category ${\mathcal E}$. 
For a pair of pivotal structures $a,a'$ on ${\mathcal E}$, 
denote by $a\boxtimes a'$ the corresponding pivotal structure on $\C$. 
Then it is easy to show that $\chi(a\boxtimes a',\M)=Z_b$ (see Example \ref{vost}), where $b=1$ iff $\overline{a}=a'$.  
Thus $a\boxtimes a'$ is matched to $\M$ if and only if $a'=\overline{a}$. 

More generally, for any finite tensor category $\C$ and indecomposable semisimple $\C$-module $\M$, regard $\M$ as a module over $\C\boxtimes \C_\M^*$. 
Let $a,a'$ be pivotal structures on $\C$ and $\C_\M^*$, such that $a$ is matched to $\M$, 
and let $a\boxtimes a'$ be their product. Then $\chi(a\boxtimes a',\M)$ is an object of the Drinfeld center ${\mathcal Z}(\C)$ which projects to $\bold 1=\chi(a,\M)\in \C_\M^*$.
Hence $\chi(a\boxtimes a',\M)=Z_b$ for some $b\in {\rm Aut}_{\otimes}({\rm Id}_{\C_\M^*})$. It is easy to show that 
$b=1$ iff $a'=\overline{a}$. Hence $a\boxtimes a'$ is matched to $\M$ 
if and only if $a'=\overline{a}$.   
\end{example} 

\subsection{Matched pivotal structures for fusion categories over fields of characteristic zero}

Let $\C$ be a fusion category over a field $\k$ of characteristic zero, and $\M$ an indecomposable semisimple $\C$-module
category. Recall that $X_r$, $r\in J$ denote the simple objects of $\C$. Let $a$ be a pivotal structure on $\C$ and 
$\dim(X_r)$ be the dimension of $X_r$ in this pivotal structure. Let  
$$
Q:=\sum_{r\in J}\dim(X_r^*)X_r\in {\rm Gr}(\C)\otimes \k.
$$
It is easy to check that $Q^2=\dim(\C)Q$, where 
$$
\dim(\C)=\sum_{r\in J}\dim(X_r)\dim(X_r^*)
$$ 
is the global dimension of $\C$, which is independent on the pivotal structure $a$ and is nonzero by \cite{EGNO}, Theorem 7.21.12. 

\begin{lemma}\label{tracelemma} (\cite{Sch1}, Proposition 5.7) If $a$ is matched to $\M$ then one has 
$$
\Tr(Q|_{{\rm Gr}(\M)})=\dim(\C).
$$
Thus, $Q|_{{\rm Gr}(\M)}$ has rank $1$.  
\end{lemma}  

\begin{proof} 
By \cite{EGNO}, Proposition 7.12.28, for any $M,N,P\in \M$ we have a canonical isomorphism 
$$
\underline{\Hom}_\C(N,P)\otimes M\cong {}^*\underline{\Hom}_{\C_\M^*}(M,N)\otimes P. 
$$
Recall that 
$$
\underline{\Hom}_\C(N,P)=\oplus_{r\in J}\Hom(X_r\otimes N,P)\otimes X_r
$$
Thus the above isomorphism can be rewritten as 
$$
\oplus_{r\in J}\Hom(X_r\otimes N,P)\otimes X_r\otimes M\cong \oplus_{r\in J}\Hom(X_r\otimes M,N)^*\otimes X_r^*\otimes P. 
$$
In particular, setting $N=P=M_i$, we get 
$$
\oplus_{r\in J}\Hom(X_r\otimes M_i,M_i)\otimes X_r\otimes M\cong \oplus_{r\in J}\Hom(X_r\otimes M,M_i)^*\otimes X_r^*\otimes M_i. 
$$
Thus, summing over $i$, we get 
$$
\oplus_{r\in J}\Tr(X_r|_{{\rm Gr}(\M)}) X_r\otimes M\cong \oplus_{r\in J, i\in I}\Hom(X_r\otimes M,M_i)^*\otimes X_r^*\otimes M_i. 
$$
Let us regard this as an identity in ${\rm Gr}(\M)\otimes \k$, and let $M=\sum_{i\in I} m_iM_i$, where $\bold m=(m_i)$ is the vector defining the module trace on $\M$. Then the above identity takes the form
$$
\sum_{r\in J}\Tr(X_r|_{{\rm Gr}(\M)}) \dim(X_r)M=\sum_{r\in J, i\in I}\dim(X_r)\dim\Hom(M,M_i)X_r^*M_i, 
$$
i.e., since $[X_r^*]=[X_r]^T$ in the standard basis of ${\rm Gr}(\M)$, we have
$$
\Tr(Q|_{{\rm Gr}(\M)}) M=\dim(\C)M,  
$$
which yields 
$$
\Tr(Q|_{{\rm Gr}(\M)})=\dim(\C),  
$$
as desired. 
\end{proof} 

\begin{corollary}\label{mul1} (\cite{Sch1}) The dimension character occurs with multiplicity $1$ in the ${\rm Gr}(\C)\otimes \k$-module ${\rm Gr}(\M)\otimes \k$, so the vector $\bold  m$ is uniquely determined by the equation $N_r\bold m=\dim(X_r)\bold m$. 
\end{corollary} 

\begin{proof} This follows directly from Lemma \ref{tracelemma}. 
\end{proof} 

Let $\overline{\bold m}=(\overline{m}_i)$ be a vector defining a module trace on $\M$ for the pivotal structure $\overline{a}$ on $\C$ defined in Subsection \ref{dps}. 

\begin{proposition}\label{nonze} We have $\sum_{i\in I} m_i\overline{m}_i\ne 0$.
\end{proposition} 

\begin{proof} Denote by $(,)$ the inner product on ${\rm Gr}(\M)$ in which $M_i$ are orthonormal. Let $M=\sum_{i\in I} m_iM_i$, $\overline{M}=\sum_{i\in I} \overline{m}_iM_i$. Then we have $Q^*M=\dim(\C)M$, hence
$$
(M,QM_j)=(Q^*M,M_j)=\dim(\C)(M,M_j)=\dim(\C)m_j,
$$
which implies that $QM_j=\beta_j\overline{M}$, where 
$$
(\sum_{i\in I} m_i\overline{m}_i)\beta_j=\dim(\C)m_j,
$$
implying that $\sum_{i\in I}m_i\overline{m_i}\ne 0$. 
\end{proof}

Proposition \ref{nonze} shows that we can normalize $m_i$ and $\overline{m}_i$ so that $\sum_{i\in I} m_i\overline{m}_i=\dim(\C)$, i.e., $\beta_j=m_j$. This normalization is unique up to the action of $\k^\times$ by $\bold m\to \lambda \bold m$ and $\overline{\bold m}\mapsto \lambda^{-1}\overline{\bold m}$. 
Let us make such a normalization from now on, and call it a {\it pivotal normalization}. 

Note that if $\k=\Bbb C$ then a pivotal normalization can be fixed so that $\overline{m}_i$ are complex conjugate to $m_i$, which justifies the notation. Note also that if $a$ is a spherical structure then there is a unique up to sign pivotal normalization such that $m_i=\overline{m_i}$ (so that $m_i\in \Bbb R$ if $\k=\Bbb C$). Let us call such a pivotal normalization {\it spherical}.  

\begin{proposition}\label{dims} For a pivotal normalization, one has 
$$
\dim\underline{\Hom}_\C(M_i,M_j)=\overline{m}_im_j.
$$ 
Moreover, this property characterizes pivotal normalizations. 
\end{proposition} 

\begin{proof} It is easy to see that $\dim\underline{\Hom}_\C(M_i,M_j)=(M_i,QM_j)$, but $QM_j=
m_j\overline{M}$, so $(M_i,QM_j)=\overline{m}_im_j$. The second statement is clear. 
\end{proof} 

\subsection{Pivotal structures on multifusion categories and matched pivotal structures}
\label{sec:pivot-struct-mult} 
Let $\C$ be a fusion category over a field $\k$ of characteristic zero, 
and $\M$ be a semisimple indecomposable $\C$-module category. Then 
$\D:=\C_{\C\oplus \M}^{*{\rm op}}$ is a multifusion category. Namely, we have 
$\D=\D_{11}\oplus \D_{12}\oplus \D_{21}\oplus \D_{22}$, where 
$\D_{11}=\C$, $\D_{12}=\M$, $\D_{21}=\M^\vee={\rm Fun}_\C(\M,\C)$ (which is the opposite category of $\M$ viewed as a right $\C$-module via duality), 
and $\D_{22}=\C_\M^{*{\rm op}}$. 
 
\begin{proposition}\label{mult}
(i) Any pivotal structure on $\D$ gives rise to a pivotal structure on $\C$ matched to $\M$ and a module trace on $\M$.

(ii) Let $a$ be a pivotal structure on $\C$ matched to $\M$, and fix the corresponding vectors $\bold m=(m_i)$, 
$\overline{\bold m}=(\overline{m}_i)$ using a pivotal normalization. Then the category $\D$ has a pivotal structure which on $\D_{11}$ is defined by $a$, on $\D_{22}$ by $\overline{a}$, and the left traces on $\D_{12}$ and $\D_{21}$ are given by $\bold m$, $\overline{\bold m}$, respectively. Moreover, this pivotal structure is spherical if and only if $a$ is a spherical structure and the normalization of $\bold m$, $\overline{\bold m}$ is spherical. 
\end{proposition}   

\begin{proof} (i) The pivotal structure on $\C$ is given by the restriction of the pivotal structure on $\D$ 
to $\D_{11}$, while the module trace on $\M$ is the left trace on $\D_{12}$. 

(ii) In $\D$ we have $M_j\otimes M_i^*=\underline{\Hom}_\C(M_i,M_j)$ and $M_i^*\otimes M_j=\underline{\Hom}_{\C_\M^*}(M_i,M_j)$, where $M_i^*$ is $M_i$ viewed as an object of $\M^\vee$. Therefore, the statement follows from Proposition \ref{dims}. 
\end{proof} 
 
\subsection{Existence of matched pivotal structures} 

The following question is raised in \cite{Sch1}. 

\begin{question}\label{scha} (\cite{Sch1}, end of Section 3) Let $\C$ be a fusion category over $\k$ (and ${\rm char}(\k)=0$). Let $\M$ be a semisimple indecomposable $\C$-module category. Does $\C$ admit a pivotal structure matched to $\M$? In other words, is $S^2$ a conjugation by a trivial grouplike element of the corresponding 
semisimple weak Hopf algebra $H$? 
\end{question} 

By Proposition \ref{mult}(i), a positive answer to Question \ref{scha} follows from the well known conjecture that any multifusion category admits a pivotal structure (see \cite{EGNO}, p.180). 

Note that the answer is "yes" for spherical fusion categories $\C$ when $\C_\M^*$ has no invertible objects of order $2$; in this case any spherical structure on $\C$ is automatically matched to $\M$, as $\chi(a,\M)=\bold 1$. For instance, this happens if $\C$ has an odd Frobenius-Perron dimension $d$ (i.e., such that the norm $N(d)$ of $d$ is odd). Indeed, in this case the dual category $\C_\M^*$ (which has the same dimension $d$) cannot have invertible objects of order $2$, as the order of an invertible object must divide $d$ (\cite{EGNO}, Theorem 7.17.6). 

On the other hand, if $\C$ is not semisimple, ${\rm char}(\k)=0$ and $\M=\Vec$, then a matched pivotal structure does not exist, since in the corresponding Hopf algebra $S^2\ne {\rm Id}$ (and any trivial grouplike element equals $1$ in this case). 

\section{Eigenvalues of $S^2$} 

\subsection{The case of a matched pivotal structure} 
Let $\C$ be a finite tensor category, $\M$ an indecomposable semisimple $\C$-module category, and $a$ 
a pivotal structure on $\C$ matched to $\M$. Let $H$ be the corresponding weak Hopf algebra, and 
let us compute the eigenvalues of $S^2$ on $H$.

As explained above, we have $S^2(x)=g_axg_a^{-1}$, where $g_a=yS(y)^{-1}$ is the pivotal element; in particular $S^2$ is semisimple. 
Let $m_i\in \k$ be the projections of $y$ to the summands of $A$. 

\begin{proposition} \label{eige} The characteristic polynomial of $S^2: H\to H$ has the form 
\begin{equation}
\chi(z)=\prod_{i,j,k,l\in I}(z-\lambda_{ijkl})^{n_{ijkl}}, 
\end{equation} 
with 
$$ 
\lambda_{ijkl}:=\frac{m_jm_l}{m_im_k}
$$
and 
$$ 
n_{ijkl}:=\sum_{q,r\in J}N_{qi}^jC_{qr}N_{rl}^k,
$$
where $C=(C_{qr})$ is the Cartan matrix of $\C$. 
\end{proposition} 

\begin{proof} Let $F:\C\to {\rm End}\M$ be the monoidal functor corresponding to $\M$, and $\overline{F}: \C\to \Vec$ 
its composition with the forgetful functor. Let $P_r$ be the projective cover of $X_r$. Then the functor $\overline{F}$ is represented by the object 
$P:=\oplus_{r\in J}\overline{F}(X_r^*)\otimes P_r$. Thus 
$$
H=\End \overline{F}=\End P=\oplus_{q,r\in J}\overline{F}(X_q)\otimes \Hom(P_q,P_r)\otimes \overline{F}(X_r^*).
$$
We may view $H$ as a bimodule over $A_s\otimes A_t$, using left and right multiplication by elements of $A_s$ and $A_t$. 
Let $e_i,i\in A$ be the primitive idempotents of $A_s$, so that $S(e_i)$ are the primitive idempotents of $A_t$. Then for $X\in \C$ we have 
$$
e_jF(X)e_i=\Hom(M_j,X\otimes M_i).
$$
Thus, 
$$
V_{ijkl}:=e_jS(e_k)He_iS(e_l)=
$$
$$
\oplus_{q,r\in J} \Hom(M_j,X_q\otimes M_i)\otimes \Hom(P_q,P_r)\otimes \Hom(M_l,X_r^*\otimes M_k).
$$
The operator $S^2$ acts on the space $V_{ijkl}$ by the scalar $\frac{m_jm_l}{m_im_k}$, 
and 
$$
\dim V_{ijkl}=\sum_{q,r\in J}N_{qi}^jC_{qr}N_{rl}^k=n_{ijkl}.
$$
This implies the statement since $H=\oplus_{i,j,k,l}V_{ijkl}$.  
\end{proof} 

\begin{example}\label{taft} Let $T_n$ be the Taft Hopf algebra of dimension $n^2$ generated by a grouplike element $g$ such that $g^n=1$ and 
an element $x$ such that $x^n=0$ and $gxg^{-1}=qx$, where $q$ is a primitive $n$-th root of unity (${\rm char}(\k)=0$). 
The coproduct is defined by $\Delta(x)=1\otimes x+x\otimes g$. Then $T_n$ contains the group algebra $\k[\Bbb Z/n]$ generated by $g$. 
So we have a tensor functor ${\rm Res}: \C=\Rep T_n\to \M=\Rep {\Bbb Z}/n$, in particular giving $\M$ a structure of a semisimple $\C$-module.
We have $S(x)=-xg^{-1}$, so $S^2(x)=gxg^{-1}$. This implies that the element $g$ defines a pivotal structure on $\C$. 

We claim that this pivotal structure is matched to $\M$. Indeed, this follows immediately from the fact that the tensor functor ${\rm Res}:\C\to \M$ 
admits a splitting $\M\to \C$ coming from the natural homomorphism $T_n\to \k[\Bbb Z/n]$. The corresponding numbers $m_i$ equal $q^i$ (for the natural labeling of objects). 
It is clear that the dimension character occurs in ${\rm Gr}(\M)$ with multiplicity $1$, since ${\rm Gr}(\C)={\rm Gr}(\M)$. 
Thus in this case by Proposition \ref{eige} the operator $S^2$ has eigenvalues $q^i$ on $H$, with equal multiplicities $n^3$.  
\end{example} 

\begin{corollary} If $\C$ is semisimple over a field of characteristic zero then 
the characteristic polynomial of $S^2$ on $H$ is given by Proposition \ref{eige} 
with $m_i$ determined uniquely up to scaling by the formula 
$$
N_r\bold m=\dim(X_r)\bold m.
$$
\end{corollary}  

\begin{proof}
This follows from Proposition \ref{eige} and Corollary \ref{mul1}. 
\end{proof} 

\begin{example} Let $\C$ be a pivotal fusion category and $\M=\C$. Then $m_i=d_i$, where $d_i=\dim(X_i)$. 
Thus the characteristic polynomial of $S^2$ on $H$ is
$$
\chi(z)=\prod_{i,j,k,l\in J}(z-\lambda_{ijkl})^{n_{ijkl}},
$$
where $n_{ijkl}=\dim\Hom(X_j\otimes X_k,X_i\otimes X_l)$ and $\lambda_{ijkl}=\frac{d_jd_k}{d_id_l}$ is the eigenvalue of $S^2$ 
on $\Hom(X_j\otimes X_k,X_i\otimes X_l)$. 
\end{example} 

Note that this polynomial does not depend on the choice of a pivotal structure.
Indeed, any two pivotal structures $a,a'$ differ by a tensor automorphism $b$ of the identity functor, 
so the corresponding dimensions are related by the formula $d_i'=d_ib_i$, where $b_i=b|_{X_i}$. 
But if $n_{ijkl}\ne 0$ then $b_jb_k=b_ib_l$, as desired. 

\subsection{The case of a general fusion category over a field of characteristic zero}

Let us now compute the eigenvalues of $S^2$ on $H$ in the case when $\C$ is any fusion category 
over a field $\k$ of characteristic zero and $\M$ is any indecomposable semisimple $\C$-module category
(without assuming that $\C$ has a pivotal structure). Recall that we have a multifusion category 
$\D=\C_{\C\oplus \M}^{*{\rm op}}$. Let $\widetilde{\D}$ be its pivotalization, defined \cite{EGNO}, 
Subsection 7.21. Then $\widetilde{\D}$ is a spherical category, with $\widetilde{\D}_{11}=\widetilde{\C}$ (the pivotalization of $\C$) and $\widetilde{\D}_{12}=:\widetilde{\M}$, which is an indecomposable semisimple $\widetilde{\C}$-module category. 

\begin{definition} We will call $\widetilde{\M}$ the pivotalization of $\M$.
\end{definition}

Note that $\widetilde{\M}$ has a canonical forgetful functor $\widetilde{\M}\to \M$ compatible with the forgetful functor $\widetilde{\C}\to \C$. Also, by Proposition \ref{mult1}(i), 
the canonical pivotal structure on $\widetilde{\C}$ is matched to $\widetilde{\M}$. 

For convenience assume that $\k=\Bbb C$ (this is not essential). Let $X_r^\pm$ be the simple objects of $\widetilde{\C}$
of positive, respectively negative dimensions which map to $X_r\in \C$ under the forgetful functor. Similarly, 
let $M_i^\pm$ be the simple objects of $\widetilde{\M}$ of positive, respectively negative dimensions which map to $M_i\in \M$ under the forgetful functor. 

Let $\nu_i=|M_i|^2>0$ be the M\"uger squared norms of the simple objects of $\M=\D_{12}$ 
(recall that they can be defined for any multifusion category, see \cite{EGNO}, Definition 7.21.2). 
Then $m_i^\pm=\dim M_i^\pm$ are given by the formula 
$$
m_i^\pm=\pm\sqrt{\nu_i}. 
$$

Let $N_{ri}^{j\pm}=\dim\Hom(X_r^+\otimes M_i^+,M_j^\pm)$. 
Let 
$$
n_{ijkl}^+:=\sum_{r\in J}(N_{ri}^{j+}N_{rl}^{k+}+N_{ri}^{j-}N_{rl}^{k-}),
$$
$$
n_{ijkl}^-:=\sum_{r\in J}(N_{ri}^{j-}N_{rl}^{k+}+N_{ri}^{j+}N_{rl}^{k-}).
$$
Let 
$$
\lambda_{ijkl}=\sqrt{\frac{\nu_j\nu_k}{\nu_i\nu_l}}.
$$

One of our main results is the following theorem. 

\begin{theorem}\label{main} The characteristic polynomial of $S^2$ on $H$ is 
$$
\chi(z)=\prod_{i,j,k,l\in I}(z-\lambda_{ijkl})^{n_{ijkl}^+}(z+\lambda_{ijkl})^{n_{ijkl}^-}.
$$
\end{theorem} 

\begin{proof}
Let $\widetilde{H}$ be the weak Hopf algebra attached to $(\widetilde{\C},\widetilde{\M})$. 
We have 
$$
\widetilde{H}=\oplus_{i,j,k,l,s_1,s_2,s_3,s_4} V_{ijkl}^{s_1s_2s_3s_4},
$$ 
where $s_p=\pm$ for $p=1,2,3,4$ and 
$$
V_{ijkl}^{s_1s_2s_3s_4}:=\oplus_{r\in J, s=\pm} \Hom(M_j^{s_1},X_r^s\otimes M_i^{s_2})\otimes \Hom(M_l^{s_3},(X_r^s)^*\otimes M_k^{s_4}).
$$
Moreover, it is easy to see that the space $V_{ijkl}^{s_1s_2s_3s_4}$ depends up to a natural isomorphism 
only on $s=s_1s_2s_3s_4$, so we will denote it by $V_{ijkl}^s$. We have $\dim V_{ijkl}^s=n_{ijkl}^s$. 
By Proposition \ref{eige}, $S^2$ acts on $V_{ijkl}^+$ by the scalar
$\lambda_{ijkl}$ and on $V_{ijkl}^-$ by $-\lambda_{ijkl}$.  

Now let 
$$
V_{ijkl}=\oplus_{r\in J} \Hom(M_j,X_r\otimes M_i)\otimes \Hom(M_l,X_r^*\otimes M_k).
$$
Then $H=\oplus_{i,j,k,l}V_{ijkl}$, and we have a natural isomorphism $V_{ijkl}\cong V_{ijkl}^+\oplus V_{ijkl}^-$. 
Moreover, it is easy to show that the squared antipode $S^2$ of $H$ acts by $\lambda_{ijkl}$ on $V_{ijkl}^+$ and by $-\lambda_{ijkl}$ on 
$V_{ijkl}^-$. This implies the theorem. 
\end{proof}
 
\section{Eigenvalues of $S^2$ for dynamical quantum groups at roots of $1$}

\subsection{The $\mathfrak{sl}_2$-case}
If the category $\C$ is not semisimple, then the dimension character may occur more than once in the ${\rm Gr}(\C)\otimes \k$-module ${\rm Gr}(\M)\otimes \k$, 
which leads to the non-uniqueness (even up to scaling) of a solution $\bold m$ of the equations $N_r\bold m=\dim(X_r)\bold m$. 

As an example consider 
$\C=\Rep \mathfrak{u}_q(\mathfrak{sl}_2)$ over $\Bbb C$, where $\mathfrak{u}_q(\mathfrak{sl}_2)$ is the small quantum group at a primitive $\ell$-th root of unity $q=e^{2\pi is/\ell}$, $-\ell/2<s<\ell/2$, with $\ell$ odd.
Namely, $\mathfrak{u}_q(\mathfrak{sl}_2)$ is the Hopf algebra generated by $E,F,K^{\pm 1}$ with defining relations 
$$
KE=q^2EK,\ KF=q^{-2}FK,\ EF-FE=\frac{K-K^{-1}}{q-q^{-1}}, 
$$
$$
E^\ell=F^\ell=K^\ell-1=0
$$
with coproduct given by 
$$
\Delta(E)=E\otimes 1+K\otimes E,\ \Delta(F)=F\otimes K^{-1}+1\otimes F,\ \Delta(K)=K\otimes K,
$$
and the counit and antipode given by 
$$
\varepsilon(E)=\varepsilon(F)=0,\ \varepsilon(K)=1,\ S(E)=-K^{-1}E,\ S(F)=-FK,\ S(K)=K^{-1}.
$$
Thus, $S^2(x)=K^{-1}xK$, so $\C$ has a pivotal structure defined by the element $K^{-1}$. 
In this case, ${\rm Gr}(\C)$ has a basis $X_1=\bold 1,...,X_\ell$, where 
$X_i$ is the irreducible $i$-dimensional representation, and 
the tensor product rule for $X_i$ is given in \cite{EGNO}, Subsection 5.8: 
$$
X_2X_i=X_{i+1}+X_{i-1},\ i\le \ell-1; X_2X_\ell=2X_1+2X_{\ell-1}.
$$
Thus ${\rm Gr}(\C)$ is generated by $X=X_2$, and we have 
$X_j=P_j(X)$, where $P_j$ is  the Chebysheff polynomial of the second kind defined by $P_j(2\cos\theta)=\frac{\sin j\theta}{\sin \theta}$
(note that it has integer coefficients). This implies that $Q(X)=0$, where 
$$
Q(X)=XP_\ell(X)-2P_{\ell-1}(X)-2.
$$
Writing $X=z+z^{-1}$, we get 
$$
(z+z^{-1})(z^{\ell-1}+...+z^{-\ell+1})-2(z^{\ell-2}+...+z^{-\ell+2})-2=0,
$$
i.e., 
$$
z^\ell+z^{-\ell}-2=0.
$$
This polynomial has double roots at all $\ell$-th roots of $1$, 
which means that $Q$ has double roots at $\alpha_j=2\cos(\pi ij/\ell)$, $j=1,...,(\ell-1)/2$, and a simple root at $\alpha_0=2$, i.e., 
$$
Q(x)=(x-\alpha_0)\prod_{i=1}^{\frac{\ell-1}{2}}(x-\alpha_j)^2,
$$
So ${\rm Gr}(\C)=\Bbb Z[x]/(Q)$ and the categorical dimension character corresponds to evaluating at $x=\alpha_s$ (where $\alpha_s=\alpha_{-s}$ if $s<0$), while the Frobenius-Perron dimension character corresponds to evaluating at $x=\alpha_0$. 

Recall that we have a cyclic group $\Bbb Z/\ell$ inside $\mathfrak{u}_q(\mathfrak{sl}_2)$ generated by the element $K$. 
Thus, ${\rm Gr}(\C)$ acts naturally on ${\rm Gr}(\Rep \Bbb Z/\ell)=\Bbb Z[\chi]/(\chi^\ell-1)$ by 
$x\circ f=(\chi+\chi^{-1})f$. It is easy to see that this module is just the associated graded of ${\rm Gr}(\C)$
with respect to the (2-step) radical filtration. In particular, the categorical dimension character (unlike the Frobenius-Perron dimension character) 
occurs in ${\rm Gr}(\M)\otimes \Bbb C$ with multiplicity $2$. 

Let us compute the possible vectors $\bold m$. The equations for the coordinates $m_i$ of $\bold m$ are 
$$
m_{j+1}+m_{j-1}=(q+q^{-1})m_j,
$$
where $j\in \Bbb Z/\ell$. Hence 
$$
m_j=aq^j-bq^{-j},
$$
$a,b\in \Bbb C$. Moreover, we must have $m_j\ne 0$ for all $j$, hence $\Lambda^\ell\ne 1$, where $\Lambda=a/b\in \Bbb C\Bbb P^1$. 
Thus, up to scaling we have a 1-parameter family of possibilities parametrized by $\Lambda$. 

It turns out that all these possibilities are realized. Namely, let $\Lambda\in \Bbb C^*$ be not an $\ell$-th root of $1$, and 
let ${\mathcal J}_\Lambda(\lambda)$ be the dynamical twist defined in \cite{EN},
Example 5.1.10, and let $H_{{\mathcal J}_\Lambda}=H_\Lambda$ be the corresponding weak Hopf algebra (\cite{EN}, Subsection 4.2). 
Namely, $H_\Lambda=\End A\otimes \mathfrak{u}_q(\mathfrak{sl}_2)$ as an algebra, where $A=\Bbb C[\Bbb Z/\ell]$ is the base,
with coproduct twisted by ${\mathcal J}_\Lambda$. It corresponds to a $\C$-module category $\M_\Lambda$ such that ${\rm Gr}(\M_\Lambda)={\rm Gr}(\M)$ 
as a ${\rm Gr}(\C)$-module. By Theorem 5.3.2 of \cite{EN}, this weak Hopf algebra is self-dual, i.e., 
$H_\Lambda\cong H_\Lambda^{*{\rm cop}}$, so that $\C_\M^{*{\rm op}}\cong \C^{\rm op}\cong \C$ (as $\C$ is braided). Since $\C$ does not have nontrivial invertible objects, 
we get $\chi(a,\M_\Lambda)=\bold 1$ (where $a$ is the unique ribbon pivotal structure on $\C$), i.e., $a$ is matched to 
$\M_\Lambda$. Thus, $S^2$ on $H_\Lambda$ is the conjugation by the trivial grouplike element $g=yS(y)^{-1}$,
where $y\in A_s\cong \Bbb C[\Bbb Z/\ell]$. 

Let us compute $g$. 

\begin{proposition}\label{sl2prop} 
One has 
$$
g=yS(y)^{-1},
$$
where $y=(y_1,...,y_\ell)\in A_s$  and 
$$
y_i=\Lambda q^i-q^{-i}. 
$$
In particular, the characteristic polynomial of $S^2$ on $H_\Lambda$ is 
$$
\chi(z)=\prod_{i,j,k,l\in \Bbb Z/\ell}\left(z-\frac{(\Lambda q^j-q^{-j})(\Lambda q^k-q^{-k})}{(\Lambda q^i-q^{-i})(\Lambda q^l-q^{-l})}\right)^{\ell}.
$$ 
\end{proposition} 

The proof of Proposition \ref{sl2prop} is given in the next subsection. 

Thus, unlike the semisimple case, the eigenvalues of $S^2$ are nontrivial functions of the continuous parameter $\Lambda$. 

In particular, consider the limiting points $\Lambda=0,\infty$. At these points, there exist limiting weak Hopf algebras $H_0$, $H_\infty$ 
which correspond to the module category $\M=\M_0={\rm Rep}(\Bbb Z/\ell)$ over $\mathfrak{u}_q(\mathfrak{sl}_2)$ and a similar category $\M_\infty$ for $q$ replaced by $q^{-1}$, obtained by twisting of $\M_0$ by the R-matrix of $\mathfrak{u}_q(\mathfrak{sl}_2)$.
In these cases, the characteristic polynomial of $S^2$ on $H_\Lambda$ is 
$$
\chi(z)=(z^\ell-1)^{\ell^4}.
$$

\subsection{The case of an arbitrary simple Lie algebra} 
Proposition \ref{sl2prop}  is a special case of Proposition \ref{gprop} below, which applies to any simple Lie algebra $\g$ over $\Bbb C$. Namely, let $\ell$ 
be an odd positive integer coprime to the determinant of the Cartan matrix of $\g$. Let $T$ be a maximal torus of the simply connected group $G$ 
corresponding to $\g$, $\Bbb T\subset T$ the subgroup of elements of order dividing $\ell$, 
and $\Lambda\in T$. Let $q$ be a primitive $\ell$-th root of unity. Let $\mathfrak{u}_q(\g)$ be the corresponding small quantum group, \cite{L}, with the coproduct convention of \cite{EN}, Subsection 5.1.\footnote{Note that 
in the definition of the small quantum group in \cite{EN}, 5.1, the relations $E_\alpha^\ell=F_\alpha^\ell=0$ are accidentally omitted for non-simple roots $\alpha$. However, 
this does not affect the arguments in \cite{EN}.}  
For a positive root $\alpha$ of $\g$, let $\Lambda_\alpha=\alpha(\Lambda)$. Assume that $\Lambda_\alpha^\ell\ne 1$ for any $\alpha$. Then we can define the dynamical twist ${\mathcal J}_\Lambda(\lambda)$ (\cite{EN}, Proposition 5.1.9)
and let $H_{{\mathcal J}_\Lambda}=H_\Lambda$ be the corresponding weak Hopf algebra (\cite{EN}, Subsection 4.2). 
Namely, $H_\Lambda=\End A\otimes \mathfrak{u}_q(\g)$ as an algebra, where $A=\Bbb C[\Bbb T]$ is the base,
with coproduct twisted by ${\mathcal J}_\Lambda$. It corresponds to a $\C$-module category $\M_\Lambda$ such that ${\rm Gr}(\M_\Lambda)={\rm Gr}(\M)$, 
as a ${\rm Gr}(\C)$-module, where $\M={\rm Rep}(\Bbb T)$. By Theorem 5.3.2 of \cite{EN}, this weak Hopf algebra is self-dual, i.e., 
$H_\Lambda\cong H_\Lambda^{*{\rm op}}$. Since $\C$ does not have nontrivial invertible objects, 
we get $\chi(a,\M_\Lambda)=\bold 1$ (where $a$ is the ribbon pivotal structure on $\C$ defined by the element $q^{-2\rho}$, with $\rho$ being the half-sum of positive roots of $\g$), i.e., $a$ is matched to 
$\M_\Lambda$. Thus, $S^2$ on $H_\Lambda$ is the conjugation by the trivial grouplike element $g=yS(y)^{-1}$,
where $y\in A_s\cong \Bbb C[\Bbb T]$. 

\begin{proposition}\label{gprop} 
One has 
$$
g=yS(y)^{-1},
$$
where $y=(y_\lambda)\in A_s$, $\lambda\in \Bbb T^\vee$  and 
$$
y_\lambda=\prod_{\alpha>0}(\Lambda_\alpha q^{(\lambda,\alpha)}-q^{-(\lambda,\alpha)}), 
$$
where $(,)$ is the inner product on $\Bbb T^\vee$ with values in $\Bbb Z/\ell$ defined by the Catran matrix of $\g$. 
In particular, the characteristic polynomial of $S^2$ on $H_\Lambda$ is 
$$
\chi(z)=\prod_{\lambda,\mu,\nu,\kappa\in \Bbb T^\vee}\left(z-\prod_{\alpha>0}\frac{(\Lambda_\alpha q^\lambda-q^{-\lambda})(\Lambda_\alpha q^\kappa-q^{-\kappa})}{(\Lambda_\alpha q^\mu-q^{-\mu})(\Lambda_\alpha q^\nu-q^{-\nu})}\right)^{\ell^{\dim \g-2{\rm rank \g}}}.
$$ 
\end{proposition} 
 
\begin{proof} Let $S_\Lambda$ be the antipode of $H_\Lambda$. By \cite{EN}, Remark 4.2.5, 
$$
S_\Lambda(x)=v^{-1}S(x)v,
$$
where
$$
v=\sum_{\lambda,\mu}E_{\lambda+\mu,\lambda}\otimes \mathcal {P}(\lambda)P_\mu
$$
and $\mathcal{P}=S(\mathcal{J}_\Lambda^{(1)})\mathcal{J}_\Lambda^{(2)}$
(using the notation of \cite{EN}). This implies that 
$$
S^2(x)=v^{-1}S(v)S^2(x)S(v)^{-1}v,
$$
and 
$$
v^{-1}S(v)=\sum_{\lambda,\mu}E_{\lambda \lambda}\otimes \mathcal P(\lambda)^{-1}S({\mathcal P})(\lambda+h)P_\mu
$$
We have $g=q^{-2\rho}v^{-1}S(v)$. Thus, the proposition follows from Proposition 2.13 in \cite{EV} 
specialized to roots of unity, taking into account that the coproduct in the quantum group used in \cite{EV} is opposite to that of \cite{EN} (hence $q^{2\rho}$ should be replaced by $q^{-2\rho}$). 
 \end{proof} 

\newpage

\appendix

\section{Module traces and inner-product structures}

\begin{center}
\bf  by Gregor Schaumann
\end{center}

\subsection{Serre functors}

In \cite{Sch2} a  different  compatibility between pivotal tensor categories and module categories is investigated. 
The goal of this appendix is to relate module traces and these so called inner-product structures. 

 Throughout this section let $\C$ be a finite tensor category over $\k$.
The right dual $x^{*}$ of $x \in \C$ has the coevaluation 
$\be \rightarrow x \otimes x^{*}$ and $\C$ has the dual distinguished invertible object $D$ satisfying $x^{****}\cong D \otimes x \otimes D^{-1} $, see \cite[Sec. 7.19]{EGNO}.
Let $\CM$ be an exact  left $\C$ module category.
The inner Hom $\IHom(-,-)$ of $\mathcal{M}$ is 
defined via a family of natural isomorphisms $$\Hom_{\C}(x,\IHom(m,n)) \,{\cong}\, 
\Hom_{\calm}( {x \otimes m},n)$$ 
for $m,n\in \calm$ and $x \in \C$. This naturally defines a left exact 
functor $\IHom(-,-)\colon \calmopp \times \calm \to \C$. Since the module category $\calm$ is an
exact $\C$-module category,  the functor $\IHom(-,-)$ is also right exact
\cite[Cor.\,3.15\,\&\,Prop.\,3.16]{EO}.  
 
This implies the existence of the following functors \cite[Def. 4.21]{FScS1}:
\begin{definition}
Let $\calm$ be a left $\C$-module. A \emph{right relative Serre functor} on $\calm$
is an en\-do\-functor $\Sr_\calm$ of $\calm$ together with a family 
of isomorphisms
\begin{equation}
   \IHom(m,n)^{*} \xrightarrow{~\cong~} \IHom(n,\Sr_\calm(m)), 
  \label{eq:Serref}
\end{equation}
which are  natural in $m,n\in\calm$.

Analogously, a \emph{left relative Serre functor} $\Sl_\calm$ comes with a family
 of natural isomorphisms
\begin{equation}
   {}^{*}\IHom(m,n) \xrightarrow{~\cong~} \IHom(\Sl_\calm(n),m).
  \label{eq:leftSerre}
  \end{equation}
\end{definition}
The relative Serre functors are twisted module functors \cite[Lemma. 4.22]{FScS1}, i.e. there 
are coherent natural isomorphisms 
\begin{equation}
 \Sl_\calm(x \otimes m) \,{\cong}\, {}^{**}x \,\otimes \, \Sl_{\calm}(m) 
  \label{eq:Sr(am)=avv.Sr(m)}
  \end{equation}
and
$ \Sr_\calm(x \otimes m) \,\cong\, x^{**} \otimes \, \Sr_\calm(m) \,. $

The Serre functors are mutually inverse: It follows from the definition that there is a chain of natural isomorphisms 
\begin{equation}
  \label{eq:inverse}
  \IHom(\Sl \Sr(n),m) \cong {}^{*}\IHom(m,\Sr(n)) \cong \IHom(n,m),
\end{equation}
for all $n,m \in \mathcal{M}$. Thus
we conclude
 from the enriched Yoneda lemma (see e.g. \cite[Lemma 4.11]{Sch2}) that $\Sl \Sr \cong \id_{\calm}$. Similarly, 
it follows that 
$\Sr \Sl \cong \id_{\calm}$.

If $\calm$ is additionally semisimple, it is also an exact module category over $\Vect$ and thus there exists moreover the \emph{$\k$-Serre functor} $\Sk$ on $\mathcal{M}$ which is defined via 
\begin{equation}
  \label{eq:k-Serre}
  \Hom_{\mathcal{M}}(m,n) \cong \Hom_{\mathcal{M}}( \Sk(n),m)^{*}.
\end{equation}
Then also $\Sk$ is an invertible twisted module functor with 
\begin{equation}
  \label{eq:twist-Sk}
  \Sk(a \otimes m) \cong a^{* *} \otimes \Sk(m).
\end{equation}
The basic relation between the two Serre functors is obtained from the action of  the dual distinguished invertible object $D$ of $\C$:
\begin{proposition}
  \label{proposition:rel-Serre}
 There is a canonical isomorphism of module functors 
\begin{equation}
  \label{eq:iso-module}
  \Sk \cong D \otimes \Sl.
\end{equation}
\end{proposition}
\begin{proof}
  From \cite[Eq (4.25)]{FScS1}, we obtain that there is a module natural isomorphism $\Sk \cong N^{l}$, where $N^{l}$ is the left exact Nakayama functor. 
On the other hand, Theorem 4.25 in loc. cit. shows that $N^{l} \cong D \otimes \Sl_{\mathcal{M}}$ as module functors. This concludes the proof.
\end{proof}

The Serre functors allow to describe the double adjoints of module functors. 
Denote by ${}^{\vee}F, F^{\vee}$ the right and left adjoints of a functor $F$, respectively. They are canonically module functors, if $F$ is a module functor. 
\begin{proposition}
  \label{proposition:double-dual-module}
Let $F: \CM \rightarrow \CN$ be a 
module
functor between exact module categories. 
There is a  natural module isomorphism 
\begin{equation}
  \label{eq:left-right-adj}
  F^{\vee} \cong \Sl_{\mathcal{M}} \circ {}^{\vee}F \circ \Sr_{\mathcal{N}},
\end{equation}
which is coherent with respect to the composition of module functors. 
\end{proposition}
\begin{proof}
By \cite{EO}, $F$ is exact and thus its adjoints exist. 
  The isomorphism is defined by 
  \begin{equation}
    \label{eq:proof-lr}
    \begin{split}
      \IHom(F^{\vee}(n),m) & \cong \IHom(n,F(m)) \cong {}^{*}\IHom(Fm, \Sr_{\mathcal{N}}(n))\\
&\cong {}^{*}\IHom(m,{}^{\vee}F \circ \Sr_{\mathcal{N}}(n))\\
& \cong \IHom(\Sl_{\mathcal{M}} \circ {}^{\vee}F \circ 
\Sr_{\mathcal{N}} (n),m)).
    \end{split}
  \end{equation}
using the enriched Yoneda lemma. 
That it is a module natural isomorphism as well as its coherence follow from standard arguments. 
\end{proof}
Thus it follows that 
$(F^{\vee})^{\vee} \cong \Sl_{\mathcal{N}} \circ F \circ \Sr_{\mathcal{M}}$.

Using moreover that on $\CCm$, $\Sr_{\C}=(-)^{* *}$, we conclude that
\begin{corollary}
  \label{cor:fromCtoM}
Let $m \in \mathcal{M}$ and $F_{m}: \CCm \rightarrow \CM$ the corresponding module functor 
$F_{m}(a)=a \otimes m$.
The double dual of $F_{m}$ computes as 
\begin{equation}
(F_{m}^{\vee})^{\vee}   \cong F_{\Sl_{\mathcal{M}}(m)}.
\end{equation}
\end{corollary}
For a fixed finite tensor category $\C$ there is a 
bicategory $\mathsf{Mod}^{\mathrm{ex}}(\C)$ with objects exact module categories and morphism categories 
the categories of module functors. Every 1-morphism in  $\mathsf{Mod}^{\mathrm{ex}}(\C)$ has a left and right dual by \cite{EO} and  Proposition \ref{proposition:double-dual-module}
 allows to describe the double dual 2-functor $(-)^{* *}$ on  $\mathsf{Mod}^{\mathrm{ex}}(\C)$, which is the identity on objects and the double left adjoint on functors and natural transformations, via the Serre functors: 
 There is a  2-functor $\mathsf{S}$ on $\mathsf{Mod}^{\mathrm{ex}}(\C)$  that is defined as the identity on objects and it maps a module functor $F: \CM \rightarrow \CN$ to 
$\mathsf{S}(F)= \Sl_{\mathcal{N}} \circ F \circ \Sr_{\mathcal{M}}$. On module natural transformations it is defined analogously.  

 \begin{corollary}
\label{cor:2-fun}
   The 2-functor $(-)^{* *}$ on $\mathsf{Mod}^{\mathrm{ex}}(\C)$  is canonically isomorphic to the 2-functor $\mathsf{S}$.
 \end{corollary}

\bigskip
Assume now that $\C$ has a pivotal structure. Then the functors 
$\Sl, \Sr, \Sk$ are canonical invertible module endofunctors of $\CM$ and it is natural to consider compatibilities between the pivotal structure and  $\CM$ as studied in 
\cite{Sch1}, \cite{Sch2} and the current article.

\begin{definition}
Let $(\C,a)$ be a pivotal finite tensor category and $\CM$ an exact module category. 
  \begin{enumerate}
  \item An \emph{inner-product module structure} \cite[Def. 5.2]{Sch2} on $\CM$ is a module natural isomorphism 
$\Sl \cong \id_{\calm}$.  
\item Assume that $\mathcal{M}$ is semisimple. A \emph{module trace } on $\mathcal{M}$
 is a module natural isomorphism $\Sk \cong \id_{\mathcal{M}}$.
  \end{enumerate}
\end{definition}
It is easy to see that this is precisely the definition of module trace of the current article.

\subsection{Relating the two structures}

We would like to  compare these structures in the case where they both apply, namely 
if $\calm$ is semisimple.
If $\C$ is  additionally 
a pivotal  fusion category, it is a consequence 
of \cite[Thm 5.23]{Sch2} that an 
inner-product structure on $\CM$ is the same as a module trace on $\CM$. 

We now clarify the difference between these structures futher using 
the multi-tensor category 
$\mathcal{D}=\End_{\C}(\CCm \oplus \CM)$ as in Section \ref{sec:pivot-struct-mult}.

In Proposition \ref{mult}  it is shown that for semisimple $\C$ 
and $\mathcal{M}$ there is a bijection between 
pivotal structures on $\mathcal{D}$ and pairs of a  pivotal structure on $\C$
and a module trace on $\mathcal{M}$.
A generalization of this correspondence to exact module categories  over finite tensor categories is:
\begin{proposition}
\label{proposition:inner-prod-piv}
  Let $\CM$ be an exact module category over a finite tensor category. 
There is a bijection between pivotal structures on $\mathcal{D}=\End_{\C}(\CCm \oplus \CM)$
and pairs of  
a pivotal structure
 on $\C$ and an inner-product module structure on $\mathcal{M}$.
\end{proposition}
\begin{proof}
First note that a pivotal structure on $\mathcal{D}$ is the same as 
a pivotal structure on the full sub-bicategory $\mathsf{D}_{\C,\mathcal{M}}$ of $\mathsf{Mod}^{\mathrm{ex}}(\C)$ which has two  
objects $\CCm$ and $\CM$ and 
all higher
morphisms of $\mathsf{Mod}^{\mathrm{ex}}(\C)$ between them. 

Assume now that $\C$ is pivotal and $\CM$
has the structure of an inner-product module category. 
It follows directly from \cite[Thm 5.13]{Sch2}
that   $\mathsf{D}_{\C,\mathcal{M}}$   is a pivotal bicategory. 

On the other hand, if  $\mathsf{D}_{\C,\mathcal{M}}$  is pivotal, 
we first obtain a pivotal structure on $\C$. 
 Consider next the module functors $F_{m}: \CCm \rightarrow \CM$ , $c \mapsto c \otimes m$ 
for $m \in \mathcal{M}$. Module natural transformations between $F_{m}$ and $F_{n}$ for $m,n \in \mathcal{M}$
are the same as morphisms from $m$ to $n$. By Corollary \ref{cor:fromCtoM}, the double left dual of $F_{m}$ is given by $ F_{\Sl_{\mathcal{M}}(m)}$, so the pivotal structure on $\mathcal{D}$ gives in 
particular a trivialisation of $\Sl_{\mathcal{M}}$, 
 thus an inner-product structure. 

The two constructions are mutually inverse. 
\end{proof}

To describe the analoguos correspondence for module traces on exact module categories, 
we define 
a tensor functor $\phi: \mathcal{D} \rightarrow \mathcal{D}$
 from the action of the dual distinguished object $D$  as follows.  
Denote by $\C^{\op}$ the tensor category with reversed tensor product and by $\C^{\vee}$ the dual category with reversed arrows. 
\begin{itemize}
\item On $\C \cong \End_{\C}(\CCm)^{\op}$, $\phi$  is defined by 
  \begin{equation}
    \phi(c)= D^{-1} \otimes c \otimes D,
  \end{equation}
\item on $\End_{\C}(\CM)$, $\phi$ is the identity,
\item on $\Fun_{\C}(\CCm, \CM)$, $\phi$ is defined by 
$F_{m} \mapsto  F_{D^{-1} \otimes m}$ for $F_{m} \in \Fun_{\C}(\CCm, \CM) $ as defined above,
\item on $\Fun_{\C}(\CM , \CCm) $ we define $\phi$ by 
$\phi(G_{n})= G_{D^{-1} \otimes n}$ for $n \in \mathcal{M} ^{\vee}$  and $G_{n}= \IHom(n,-) \in \Fun_{\C}(\CM , \CCm)$.
\end{itemize}
Thus $\phi$ is the quadruple dual on $\C$.

To see that $\phi$ is indeed a monoidal functor, consider the composite $G_{n} \circ F_{m} = - \otimes \IHom(n,m): \CCm \rightarrow \CCm$. We compute 
\begin{equation}
  \begin{split}
     \phi(G_{n}) \circ \phi(F_{m})(\be)&= \IHom(D^{-1} \otimes n, D^{-1} \otimes m) \\
&\simeq D^{-1} \otimes\IHom(n,m) \otimes D= \phi(G_{n} \circ F_{m})(\be).
  \end{split}
 \end{equation}
On the other hand it follows that  $\phi(F_{m}) \circ \phi(G_{n})= F_{m} \circ G_{n}$.
The remaining coherence data of the 2-functor are obvious and the axioms for a 2-functor follow directly. 
We can finally formulate:
\begin{proposition}
\label{proposition:correspon-pivD-modtr}
  Let $\C$ be a finite tensor category and $\CM$ a semisimple module category. 
There is a bijection between monoidal natural isomorphisms $(-)^{* * } \cong \phi$  on  $\mathcal{D}=\End_{\C}(\CCm \oplus \CM)$
and pairs of a pivotal structure on $\C$ and a module trace on $\mathcal{M}$.
\end{proposition}
\begin{proof}
Let $\C$ be pivotal and assume $\CM$ has a module trace. It follows from 
Proposition  \ref{proposition:rel-Serre} that the module trace induces a module natural isomorphism 
\begin{equation}
  \Sl_{\mathcal{M}} \cong D^{-1} \otimes -
\end{equation}
  of module endofunctors on $\CM$ (the pivotal structure is needed for the module structure of $D^{-1} \otimes -$). Furthermore, the pivotal structure provides a trivialisation of 
the relative Serre functor on $\CCm$ (which is just the double dual), thus 
the monoidal functor $\mathsf{S}$ on $\mathcal{D}$ from Corollary \ref{cor:2-fun} 
can be described in terms of the action of the dual distinguished object $D$. This produces 
with Corollary  \ref{cor:2-fun} a monoidal isomorphism to the functor $\phi$.

Assume to the contrary, that a monoidal isomorphism from the double dual to $\phi$ on $\mathcal{D}$ is specified. On $\C$ this produces a monoidal isomorphism from the quadruple dual to the double dual, thus a pivotal structure. 
Since the double dual of $F_{m}$ is given by $F_{\Sl_{\mathcal{M}}(m)}$, we moreover obtain 
an isomorphism $\Sl_{\mathcal{M}} \cong D^{-1} \otimes- $, which is a module natural isomorphism, thus it gives a module trace on $\CM$, 
again by  Proposition \ref{proposition:rel-Serre}.
\end{proof}

There are two main situations in which we have a bijection between  the two structures:
If $\C$ is unimodular, we use as in Section \ref{dps} the notation $\overline{a}$ for the 
``dual'' pivotal structure defined by  $\overline{a}=\iota\circ a^{-1}$ with $\iota$ the canonical isomorphism to the quadruple dual.  

First we  obtain as a consequence of Propositions \ref{proposition:inner-prod-piv} and  \ref{proposition:correspon-pivD-modtr}:
\begin{corollary}
\label{corollary:biject}
  Let $\C$ be unimodular and $\CM$ a semisimple module category. There are bijections 
  \begin{equation}
    \label{eq:1}
    \begin{split}
       \{ (\C,a) ,  \CM  \quad \text{inner product} \} \leftrightarrow &
\{ \text{pivotal structures on} \quad \D \} \\
 \leftrightarrow & 
\{ (\C,\overline{a}),  \CM  \quad \text{module trace} \}
    \end{split}
   \end{equation}
\end{corollary}
Second we consider the case that the dual category is unimodular. 
 \begin{proposition}
\label{proposition:unimod-dual}
   Assume that $\C_{\calm}^{*}$ is unimodular. Then $\CM$ has a module trace over $(\C,a)$ if and only if it has an inner-product structure over 
$(\C,a)$.
 \end{proposition}
 \begin{proof}
There is an obvious generalization of the Radford theorem, see \cite[Sec. 7.19]{EGNO}, to multitensor categories and applied to $\mathcal{D}=\End_{\C}(\CCm \oplus \CM)$  we obtain isomorphisms
\begin{equation}
  \label{eq:quadrupel-D}
  m^{****} \simeq  D_{\C}^{-1} \otimes m \otimes  D_{\C_{\calm}^{*}},
\end{equation}
for $m \in \calm$ considered under the equivalence $\Fun_{\C}(\CCm, \CM) \cong \calm$ and 
where $ D_{\C_{\calm}^{*}}, D_{\C}$ are the dual distinguished invertible objects of $\C_{\calm}^{*} $ and $\C$, 
respectively.  
Suppose that $\C_{\calm}^{*}$ is unimodular and $\CM$ is inner-product, then  using 
Proposition \ref{proposition:inner-prod-piv}, 
$m \simeq m^{****} \simeq D_{\C}^{-1}\otimes m $. Hence we obtain that  $D_{\C} \otimes - \simeq \id_{\calm}$ as $\C$-module functors and with 
 Proposition \ref{proposition:rel-Serre} we conclude that $\CM$ has a module trace. 

Conversely, if $\CM$ has a module trace, Proposition \ref{proposition:correspon-pivD-modtr} implies 
$m^{****} \simeq \phi^{2}(m)=D^{-2}_{\C} \otimes m$ and  Equation \eqref{eq:quadrupel-D} implies that  $D_{\C} \otimes -$ is isomorphic to $ \id_{\mathcal{M}}$ as module functor, thus $\CM$ is also inner-product. 
 \end{proof}

\medskip

Next we discuss  situations where the two structures disagree.
First we have as a  direct consequence  from Proposition \ref{proposition:rel-Serre}: 
\begin{corollary}
  If a semisimple module category $\CM$ over a pivotal finite tensor category $(\C,a)$
admits a module trace and the module functor $D \otimes - : \CM \rightarrow \CM$ is not isomorphic 
to the identity  functor as a linear functor, $\CM$ does not admit an inner-product module structure over any pivotal structure of $\C$. 
\end{corollary}
This situation can indeed happen:
\begin{example}
\label{example:Taft-no-ip} 
  Let $\C=\Rep T_n$ and $\calm=\Rep {\Bbb Z}/n$ be as in the Example \ref{taft}, i.e. 
$\CM$ has a module trace over $\C$. The dual distinguished invertible object $D$ of $\C$ is the one-dimensional representation with action of $g$ by $q$ and $x$ by $0$. It  
 acts non-trivially on $\Gr(\calm)$, thus $\CM$ cannot admit an inner-product structure for any pivotal structure of $\C$. 
\end{example}
In general for module categories from   tensor functors we have:
\begin{example}
  Let $F: \C \rightarrow \C'$ be a pivotal (exact) tensor functor between pivotal finite categories and let $\C'$ be semisimple. Then $\C'$ is a module category over $\C$ that is matched to the given pivotal structure. It has an inner-product structure if and only if $F(D_{\C})=\be$. 
\end{example}

We do not know of an example of the converse kind that has an inner-product structure but cannot possess a matched pivotal structure. 
 
However, reversing Proposition \ref{proposition:unimod-dual} we obtain:
\begin{corollary} 
\label{cor:conj-not-matched}
  Suppose $\C$ is unimodular and  $\C_\calm^*$ is not and suppose $(\C,a)$ is matched
to $\CM$.
Then $(\C, \overline{a})$
has an inner-product structure over $\CM$, but it cannot have a module
trace.
\end{corollary}
\begin{proof}
  If  $\CM $ has also a module trace  over  $(\C, \overline{a})$, it follows 
from Equation \eqref{eq:quadrupel-D}  that the action of $D_{\C_{\calm}^{*}}$ on $\CM$ is isomorphic to $\id_{\calm}$ as module functors, thus $\C_{\calm}^{*}$ is unimodular and we have a contradiction. 
\end{proof}

\begin{example}
  Example \ref{example:Taft-no-ip}  provides also a case where Corollary \ref{cor:conj-not-matched} applies: The category $\C \boxtimes \C_{\calm}^{*}$ is not unimodular and acts on 
$\calm$. The dual category is the center $\mathcal{Z}(\C)=\Rep(\mathfrak{u}_q(\mathfrak{sl}_2)) \boxtimes \Rep {\Bbb Z}/n$ which is unimodular and inherits a pivotal structure $a$ from $\C$. 
Then  $(\mathcal{Z}(\C), a)$ is matched to $\calm$, while $(\mathcal{Z}(\C),\overline{a})$ has an inner-product structure but is not matched. 
\end{example}

\end{document}